\documentclass[12pt]{amsart}

\usepackage[T1]{fontenc}
\usepackage[utf8]{inputenc}
\usepackage{amstext}
\usepackage{amsmath, amssymb, amsthm}
\usepackage{amssymb}
\usepackage[a4paper, total={6.2in, 9in}, centering]{geometry}
\usepackage{enumitem}
\usepackage{xcolor}
\usepackage{comment}
\usepackage{hyperref}
\usepackage{graphicx}
\usepackage{setspace}

\newtheorem{theorem}{Theorem}[section]
\newtheorem{lemma}[theorem]{Lemma}
\newtheorem{proposition}[theorem]{Proposition}
\newtheorem{corollary}[theorem]{Corollary}
\newtheorem{remark}[theorem]{Remark}

\DeclareMathOperator{\trace}{tr}

\DeclareMathOperator{\dist}{dist}

\DeclareMathOperator{\IC}{IC}
\DeclareMathOperator{\Sym}{Sym}

\title[A differential Harnack inequality]{A differential Harnack inequality for noncompact evolving hypersurfaces}
\author{Stephen Lynch}

\begin{document}

\begin{abstract}
We prove a differential Harnack inequality for noncompact convex hypersurfaces flowing with normal speed equal to a symmetric function of their principal curvatures. This extends a result of Andrews for compact hypersurfaces. We assume that the speed of motion is one-homogeneous, uniformly elliptic, and suitably `uniformly' inverse-concave as a function of the principal curvatures. In addition, we assume the hypersurfaces satisfy pointwise scaling-invariant gradient estimates for the second fundamental form. For many natural flows all of these hypotheses are met by any ancient solution which arises as a blow-up of a singularity.
\end{abstract}

\maketitle

\section{Introduction} Differential Harnack inequalities for solutions to parabolic PDE were introduced by Li and Yau in their seminal paper \cite{Li--Yau}. Let $u: M\times(0,T]\to\mathbb{R}$ denote a bounded positive solution to the heat equation, where $(M,g)$ is a compact Riemannian manifold with nonnegative Ricci curvature. The Li--Yau inequality asserts that 
    \[\partial_t u - \frac{|\nabla u|^2}{u} + \frac{n}{2t}u \geq 0.\]
This inequality is saturated by the Euclidean heat kernel. Upon integration over an appropriate path in spacetime, it recovers the classical Harnack inequality for solutions to the heat equation: for $0 < t_0 < t_1 \leq T$ we have 
    \[u(x_1, t_1) \geq \bigg(\frac{t_0}{t_1}\bigg)^{n/2}\exp\bigg( - \frac{d_{g}(x_0, x_1)^2}{4(t_1 - t_0)}\bigg)\,u(x_0,t_0).\]

Analogues of the Li--Yau differential Harnack inequality have since been found for many other equations, including geometric flows. Hamilton found remarkable Harnack inequalities for the curvature of solutions to the Ricci flow with nonnegative curvature operator \cite{Hamilton_Harnack_Ricci} (see also \cite{Brendle_Harnack}), and for weakly convex solutions of the mean curvature flow \cite{Hamilton_Harnack_MCF}. These inequalities play a fundamental role in understanding singularity formation, and have therefore had profound implications in geometry. 

Chow proved the analogue of Hamilton's Harnack inequality for compact, strictly convex hypersurfaces flowing by powers of their Gauss curvature \cite{Chow}. Andrews generalised this result to a large class of fully nonlinear flows \cite{Andrews_Harnack}. Solutions to these flows move with speed equal to a general symmetric homogeneous function of their principal curvatures. Andrews' estimate applies, in particular, when the speed function is homogeneous of degree one and inverse-concave. 

Solutions to curvature flows which arise as dilations of singularities are often noncompact, and it is desirable to have a Harnack inequality which applies to these. In the present paper we extend Andrews' differential Harnack inequality to noncompact convex hypersurfaces, provided the speed of motion is homogeneous of degree one, uniformly elliptic, and suitably  `uniformly' inverse-concave. In addition, we need to assume the hypersurfaces satisfy pointwise scaling-invariant gradient estimates. For certain flows all of these hypotheses are known to be satisfied by any ancient solution which arises as a blow-up of a singularity. For example, our Harnack inequality applies to blow-ups of compact embedded solutions to the flows introduced in \cite{Brendle_Huisken} and \cite{Lynch}.  

Let us further discuss the flow introduced in \cite{Brendle_Huisken}. There Brendle and Huisken studied domains in a Riemannian $(n+1)$-manifold whose boundaries move inward with speed equal to 
    \[\gamma(\lambda) = \bigg(\sum_{i < j} \frac{1}{\lambda_i + \lambda_j}\bigg)^{-1},\]
where the $\lambda_i$ are the principal curvatures. By implementing a surgery procedure for this flow, they were able to classify compact Riemannian manifolds with strictly two-convex boundary and nonnegative curvature in the sense that $R_{ikik} + R_{jkjk} \geq 0$. Namely, these spaces are all diffeomorphic to a standard ball or 1-handlebody. In forthcoming work with Cogo and Vi\v{c}\'{a}nek Mart\'{i}nez we completely classify the ancient solutions which can arise as blow-ups at a singularity of this flow---these are the shrinking round sphere $S^{n}$, the shrinking round cylinder $\mathbb{R}\times S^{n-1}$, and the unique rotationally symmetric translating soliton. The corresponding result for mean curvature flow is due to Brendle and Choi \cite{Brendle_Choi_a, Brendle_Choi_b}. The differential Harnack inequality proven in this paper is an important ingredient in our proof.

\subsection{Main results} Let $\Gamma \subset \mathbb{R}^n$ denote an open, symmetric (under permutations), convex cone. We assume that $\Gamma$ contains the positive cone $\mathbb{R}^n_+$. Fix a function $\gamma:\Gamma \to \mathbb{R}$ which is smooth, positive, symmetric, strictly increasing in each argument, and homogeneous of degree one.

We consider evolving immersions $F : M \times I \to \mathbb{R}^{n+1}$, $I \subset \mathbb{R}$, which satisfy the evolution equation
    \begin{equation}\label{flow}
    \partial_t F(x,t) = -G(x,t)\nu(x,t),
    \end{equation}
for every $(x,t) \in M \times I$, where $G(x,t) = \gamma(\lambda(x,t))$ and $\nu(x,t)$ is the outward unit normal. We write $\lambda(x,t)$ for the principal curvatures of $F$, i.e. the eigenvalues of the Weingarten map         \[A(X) = D_X \nu,\] 
at $(x,t)$. These will be labeled so that $\lambda_1 \leq \dots \leq \lambda_n$. A solution to \eqref{flow} is called uniformly $k$-convex if
    \[\inf_{M\times I} \frac{\lambda_1 + \dots + \lambda_k}{H} > 0,\]
where $H$ is the mean curvature. The metric induced on $M$ by the immersion $F(\cdot,t)$ will be denoted $g = g(t)$. 

The function $\gamma$ gives rise to a smooth, $O(n)$-invariant function on the space of symmetric matrices with eigenvalues in $\Gamma$. We also denote this function $\gamma$. We say that $\gamma$ is (strictly) inverse-concave if $\lambda \mapsto -\gamma(\lambda^{-1})$ is (strictly) concave on $\mathbb{R}^n_+$ or, equivalently, if $A \mapsto -\gamma(A^{-1})$ is (strictly) concave on the space of positive-definite symmetric matrices. Inverse-concavity is equivalent to the pointwise inequality
    \begin{equation}\label{IC}
    \bigg(\frac{\partial^2 \gamma}{\partial A_{ij}\partial A_{kl}}(A) + 2\frac{\partial \gamma}{\partial A_{ik}}(A)A^{-1}_{jl}\bigg)S_{ij}S_{kl} \geq 0,
    \end{equation}
for positive-definite symmetric $A$ and symmetric $S$.

We now state the Harnack inequality. As mentioned above, this is new when $M$ is noncompact. For compact $M$, the result was proven in \cite{Andrews_Harnack} under more general hypotheses---when $M$ is compact \eqref{unif ellipticity}, \eqref{pert inverse-concave} and \eqref{gradient estimate} are unnecessary. 

\begin{theorem}\label{Harnack}
Suppose $\gamma$ is inverse-concave. Let $F: M\times[0, T] \to \mathbb{R}^{n+1}$ be a complete solution to \eqref{flow} which satisfies $A \geq 0$. We assume there are positive constants $C$ and $\varepsilon$ such that at each point in $M\times[0,T]$, with respect to an orthonormal frame, we have 
    \begin{equation}\label{unif ellipticity}
    C^{-1}g_{ij} \leq \frac{\partial\gamma}{\partial A_{ij}}(A) \leq Cg_{ij}.
    \end{equation}
and 
    \begin{equation}\label{pert inverse-concave}
        \bigg(\frac{\partial^2 \gamma}{\partial A_{ij}\partial A_{kl}}(A) + 2\frac{\partial \gamma}{\partial A_{ik}}(A)(A+ \varepsilon G g)^{-1}_{jl}\bigg)S_{ij}S_{kl}\geq 0
    \end{equation}
for every symmetric $S$. In addition, we assume bounded curvature
    \[\sup_{M\times[0,T]} G < \infty\]
and the gradient estimates
    \begin{equation}\label{gradient estimate}
    \sup_{M\times[0,T]} G^{-2}|\nabla A| + G^{-3}|\nabla^2 A| < \infty.
    \end{equation}
Then, for every $(x,t) \in M\times(0,T]$ and $V \in T_x M$, we have 
    \begin{equation}\label{Harnack V}
    \partial_t G + 2\langle \nabla G, V\rangle + A(V,V) + \frac{G}{2t} \geq 0.
    \end{equation}
If $\gamma$ is convex then \eqref{pert inverse-concave} is unnecessary and instead of \eqref{gradient estimate} it suffices to assume 
    \[\sup_{M\times[0,T]} |\nabla G| + |\partial_t G| < \infty.\]
\end{theorem}

Let us comment on the hypotheses \eqref{unif ellipticity}, \eqref{pert inverse-concave} and \eqref{gradient estimate}. The evolution of the Harnack quantity contains terms involving the Hessian of $\gamma$, which need to be overcome in order to establish \eqref{Harnack V} using the maximum principle. If the solution is noncompact, we also need to be able to localise by introducing an auxhiliary function which grows at infinity. It is \eqref{unif ellipticity} and \eqref{pert inverse-concave}, and the gradient estimates \eqref{gradient estimate}, which let us achieve both of these things simultaneously. We refer to Remark~\ref{need for uniformity} for further discussion. 

As stated in Theorem~\ref{Harnack}, if $\gamma$ is convex then the assumption \eqref{pert inverse-concave} is unnecessary and \eqref{gradient estimate} can be weakened substantially. In fact, in this case the terms in the evolution of the Harnack quantity depending on the Hessian of $\gamma$ have a favourable sign. This means Hamilton's proof for the mean curvature flow applies almost verbatim. 

We now introduce natural conditions under which \eqref{unif ellipticity}, \eqref{pert inverse-concave} and \eqref{gradient estimate} are met. This leads to Theorem~\ref{Harnack k-convex} below. 

{\emph{Pointwise gradient estimates for ancient solutions.}} A solution to \eqref{flow} is called ancient if it exists for all $t  \in (-\infty, T]$. Ancient solutions are of great interest, since they arise as models for singularity formation via rescaling. 

When $\gamma$ is convex or concave, by work of Brendle and Huisken \cite{Brendle_Huisken}, a pointwise gradient estimate of the form
    \[G^{-2}|\nabla A| + G^{-3}|\nabla^2 A| \leq \Lambda(n,\gamma,C,\alpha)\]
holds for convex ancient solutions of \eqref{flow} which satisfy \eqref{unif ellipticity} and are also $\alpha$-noncollapsing.

The (interior) $\alpha$-noncollapsing property for solutions to \eqref{flow} is a form of quantitative embeddedness. Let $F(M,t) = M_t$ and suppose $M_t = \partial \Omega_t$, where $\Omega_t$ is an open subset of $\mathbb{R}^{n+1}$. We say the solution $M_t$ is $\alpha$-noncollapsing if there is a time-independent constant $\alpha > 0$ such that $\Omega_t$ admits an inscribed ball of radius $\alpha G(x,t)^{-1}$ for each $x \in M_t$. In \cite{ALM_Noncollapsing} it was shown that for compact solutions and $\gamma$ concave, the noncollapsing property is preserved forward in time. When we do not need to refer to $\alpha$ we simply say that $M_t$ is noncollapsing.

{\emph{Uniform inverse-concavity.}} For each integer $0 \leq m \leq n$, we define
    \[\Gamma_+^m = \{\sigma(\lambda) \in \mathbb{R}^{n} : \lambda = (0, \lambda_1, \dots, \lambda_m), \; \min_i \lambda_i >0, \; \sigma \in P_n\},\]
where $P_n$ is the group of permutations on $n$ elements. For $1 \leq m \leq n-1$, $\Gamma_+^m$ is the union of all of the $m$-dimensional facets of $\partial \Gamma_+^n$. Each connected component of $\Gamma_+^m$ can be identified with $\mathbb{R}_+^m$.

Since $\Gamma$ is open, convex and contains $\Gamma_+^n$, for each $1 \leq m \leq n$ we either have $\Gamma_+^m \subset \Gamma$ or else $\Gamma_+^m \cap \Gamma = \emptyset$. Let $m_*$ denote the least integer such that $\Gamma_+^m \subset \Gamma$.  For each $m_* \leq m \leq n$ we write $\gamma_m : \mathbb{R}_+^m \to \mathbb{R}$ for the function
    \[\gamma_m(\lambda_1, \dots, \lambda_m) := \gamma(0, \lambda_1, \dots, \lambda_m).\]
When $\gamma$ is strictly inverse-concave we define $m_{\IC} \geq m_*$ to be the least integer such that $\gamma_m$ is strictly inverse-concave for every $m_{\IC} \leq m \leq n$.

In Section~\ref{sec unif inverse-concavity} we demonstrate that if the eigenvalues of $A \geq 0$ satisfy
    \[\min_{0 \leq m < m_{\IC}} \dist(\tfrac{\lambda}{|\lambda|}, \Gamma_+^m) \geq \delta,\]
then \eqref{unif ellipticity} and \eqref{pert inverse-concave} hold for some positive $C = C(n,\gamma,\delta)$ and $\varepsilon = \varepsilon(n,\gamma,\delta)$. When $A \geq 0$, this is equivalent to assuming that $\lambda$ is uniformly $k$-positive with $k = n - m_{\IC} + 1$. Therefore, \eqref{unif ellipticity} and \eqref{pert inverse-concave} hold on a solution to \eqref{flow} which is uniformly $k$-convex with $k \leq n - m_{\IC} + 1$. 

As a result of all of this discussion, we have the following consequence of Theorem~\ref{Harnack}. 

\begin{theorem}\label{Harnack k-convex}
Suppose $\gamma$ is convex, or concave and strictly inverse-concave. Let $F: M\times(-\infty, 0] \to \mathbb{R}^{n+1}$ be an ancient solution to \eqref{flow}. We assume the hypersurfaces $M_t = F(M,t)$ each bound an open convex subset $\Omega_t$. In addition, for each $T < \infty$, we assume $M_t$ is noncollapsing and uniformly $k$-convex on the time interval $[-T,0]$, where $k \leq n - m_{\IC} + 1$. Finally, we assume that 
    \[\sup_{M\times[-T,0]} G < \infty\]
for each $T < \infty$. Then, for every $(x,t) \in M\times(-\infty,0]$ and $V \in T_x M$, we have 
    \begin{equation}\label{Harnack ancient}
    \partial_t G + 2\langle \nabla G, V\rangle + A(V,V) \geq 0.
    \end{equation}
\end{theorem}

\subsection{Examples} Consider the concave speeds (cf. \cite{Brendle_Huisken}) given by
    \[\gamma(\lambda) = \bigg(\sum_{i_1 < \dots < i_k} \frac{1}{\lambda_{i_1} + \dots + \lambda_{i_k}}\bigg)^{-1},\]
for $k \geq 2$ and $n \geq k+1$. In this case we may take
    \[\Gamma = \Big\{\lambda : \min_{i_1<\dots<i_k} \lambda_{i_1} + \dots + \lambda_{i_k}>0\Big\}.\]
We then have $m_{\IC} = m_* = n-k+1$, so Theorem~\ref{Harnack k-convex} applies to convex ancient solutions of \eqref{flow} which are noncollapsing and uniformly $k$-convex. 

Consider the ratios of elementary symmetric polynomials $\gamma = \sigma_k/\sigma_{k-1}$ for $k \geq 2$ and $n \geq k+1$. We may take $\Gamma$ to be the cone where $\sigma_{k} > 0$, in which case $m_* = k$ and $m_{\IC} = k + 1$ (the function $\gamma_k$ is the harmonic mean, which is inverse-concave but not strictly inverse-concave). So Theorem~\ref{Harnack k-convex} applies to convex ancient solutions which are noncollapsing and uniformly $(n-k)$-convex. 

If $\gamma : \Gamma \to \mathbb{R}$ is strictly inverse-concave, $\beta : \Gamma \to \mathbb{R}$ is inverse-concave, and $h:\mathbb{R}_+^2 \to \mathbb{R}$ is inverse-concave, then the composition 
    \[\lambda \mapsto h(\gamma(\lambda), \beta(\lambda))\]
is strictly inverse-concave. Using this observation, one finds that Theorem~\ref{Harnack k-convex} applies to convex, noncollapsing, uniformly $k$-convex ancient solutions to the flows introduced in \cite{Lynch}. This class includes all blow-up limits at a singularity of a compact embedded solution. 

\subsection{Translating solitons} Harnack inequalities are closely related to solitons. A solution to \eqref{flow} is called a translating soliton if there is a constant vector $\xi$ on $\mathbb{R}^{n+1}$ such that the hypersurfaces $M_t = F(M,t)$ satisfy $M_t = M_0 + t\xi$. Translating solitons are characterised by the identity
    \[G = -\langle \xi, \nu \rangle.\]

Notice that when $A > 0$ we have 
\[2\langle \nabla G, V\rangle + A(V,V) \geq -A^{-1}(\nabla G, \nabla G),\]
with equality for $V = - A^{-1}(\nabla G)$. In this case \eqref{Harnack ancient} becomes  
    \[\partial_t G - A^{-1}(\nabla G, \nabla G) \geq 0.\]
If equality is attained here, and $\gamma$ is strictly inverse-concave, then the solution is a translating soliton.

\begin{corollary}\label{translator}
Suppose $\gamma$ is strictly inverse-concave. Let $F : (-\infty, 0] \to \mathbb{R}^{n+1}$ be an ancient solution to \eqref{flow} such that $A > 0$. Suppose the Harnack inequality
    \[\partial_t G - A^{-1}(\nabla G, \nabla G) \geq 0\]
holds at each point in spacetime. In addition, we assume bounded curvature,
    \[\sup_{M\times[-T,0]} G < \infty,\]
and the uniform ellipticity condition \eqref{unif ellipticity} on $[-T,0]$ for each $T < \infty$. If there is a point in spacetime at which
    \[\partial_t G - A^{-1}(\nabla G, \nabla G) = 0,\]
then $F$ is a translating soliton. 
\end{corollary}

\subsection{Pointwise Harnack estimate} When $A> 0$, \eqref{Harnack V} implies a pointwise estimate comparing $G$ at different points in spacetime. Indeed, assuming $A > 0$, the inequality \eqref{Harnack V} may be restated as 
    \[\partial_t G -A^{-1}(\nabla G, \nabla G) + \frac{G}{2t} \geq 0.\]
Given times $0 < t_0 < t_1 \leq T$, integration of this inequality yields
    \[G(x_1, t_1) \geq \bigg(\frac{t_0}{t_1}\bigg)^{\frac{1}{2}} \exp\bigg(-\frac{1}{4}\inf_\gamma \int_{t_0}^{t_1} \frac{A(\dot\gamma,\dot\gamma)}{G}\,dt\bigg)\,G(x_0,t_0),\]
where the infimum is over smooth paths $\gamma : [t_0,t_1] \to M$ satisfying $\gamma(t_0) = x_0$ and $\gamma(t_1) = x_1$.

\subsection{Acknowledgements} The author is grateful to M. Langford for his valuable comments. 

\section{Uniform inverse-concavity}\label{sec unif inverse-concavity} 

In this section we establish conditions under which \eqref{unif ellipticity} and \eqref{pert inverse-concave} hold. We begin by introducing some notation. 

We write $\dot \gamma^i(\lambda)$ and $\ddot \gamma^{ij}(\lambda)$ for derivatives with respect to eigenvalues, so that 
    \[\frac{d}{ds}\bigg|_{s=0} \gamma(\lambda + s \mu) = \dot \gamma^i(\lambda)\mu_i, \qquad \frac{d}{ds}\bigg|_{s=0} \dot \gamma^i(\lambda + s\mu) = \ddot \gamma^{ij}(\lambda)\mu_j,\]
and write $\dot \gamma^{ij}(A)$ and $\ddot \gamma^{ij,kl}(A)$ for derivatives with respect to matrix entries, so that 
    \[\frac{d}{ds}\bigg|_{s=0} \gamma(A + s B) = \dot \gamma^{ij}(A)B_{ij}, \qquad \frac{d}{ds}\bigg|_{s=0} \dot \gamma^{ij}(A + sB) = \ddot \gamma^{ij,kl}(A)B_{kl}.\]
When $A$ is a diagonal matrix with entries $\lambda$, $\dot\gamma^{ij}(A)$ is also diagonal, with entries $\dot\gamma^i(\lambda)$.

Let $\operatorname{Sym}(n)$ denote the space of symmetric $n\times n$-matrices. We assume that $\gamma$ is strictly inverse-concave. That is, the function $A \mapsto -\gamma(A^{-1})$ is strictly concave for positive-definite $A \in \operatorname{Sym}(n)$. In terms of derivatives this means that
    \[(\ddot\gamma^{ij,kl}(A) + 2\dot\gamma^{ik}(A)A^{-1}_{jl})S_{ij}S_{kl} > 0\]
for every positive-definite $A \in \operatorname{Sym}(n)$ and every nonzero $S \in \operatorname{Sym}(n)$. Since $\gamma$ is homogeneous of degree one, its strict inverse-concavity is also equivalent to strict concavity of the function $A \mapsto \gamma(A^{-1})^{-1}$ in non-radial directions. In terms of derivatives,
    \[(\ddot\gamma^{ij,kl}(A) + 2\dot\gamma^{ik}(A)A^{-1}_{jl} - 2\gamma(A)^{-1}\dot\gamma^{ij}(A)\dot\gamma^{kl}(A))S_{ij}S_{kl} > 0\]
for every positive-definite $A \in \operatorname{Sym}(n)$ and every $S \in \operatorname{Sym}(n)$ which is not a multiple of $A$. 

Let $\Gamma'$ be a closed, symmetric, convex cone which is contained in the closure of $\Gamma_+^n$. We are interested in consequences of the property
    \begin{equation}\label{unif distance}
        \min_{0 \leq m < m_{\IC}} \bigg(\inf_{\lambda \in \Gamma'} \dist(\tfrac{\lambda}{|\lambda|}, \Gamma_+^m)\bigg)>0.
    \end{equation}

\begin{lemma}\label{lemma unif parabolic}
Suppose $\Gamma'$ satisfies \eqref{unif distance}. We then have 
    \[\inf_{\lambda \in \Gamma'} \dist(\tfrac{\lambda}{|\lambda|}, \partial \Gamma) > 0.\]
\end{lemma}
\begin{proof}
By definition, $\Gamma_+^m \cap \Gamma = \emptyset$ for $m < m_*$, and $\Gamma_+^m \subset \Gamma$ for $m \geq m_*$. It follows that
    \begin{equation}\label{boundary intersect positive}
        \partial \Gamma \cap \bar \Gamma_+^n = \bigcup_{0 \leq m < m_*} \Gamma_+^m.
    \end{equation}
    
The property \eqref{unif distance} implies 
    \begin{equation}\label{dist to m*}
        \min_{0 \leq m < m_*} \bigg(\inf_{\lambda \in \Gamma'} \dist(\tfrac{\lambda}{|\lambda|}, \Gamma_+^m)\bigg) > 0.
    \end{equation}
Indeed, we have $m_* \leq m_{\IC}$ by definition. Define
    \[\delta := \inf_{\lambda \in \Gamma'} \dist(\tfrac{\lambda}{|\lambda|}, \partial \Gamma).\]
The claim is that $\delta > 0$. There is a sequence $\lambda^{(k)} \in \Gamma'$ such that $|\lambda^{(k)}| = 1$ and $\dist(\lambda^{(k)}, \partial \Gamma) \to \delta$ as $k \to \infty$. Since $\Gamma'$ is closed, we may assume $\lambda^{(k)}$ converges to some $\lambda \in \Gamma'$. Consider the possibility that $\lambda \in \partial \Gamma$. In this case \eqref{boundary intersect positive} implies $\lambda \in \Gamma_+^m$ for some $1 \leq m < m_*$. But due to \eqref{dist to m*} this is impossible. So we must have $\lambda \in \Gamma$, and hence $\delta > 0$.
\end{proof}

Given a symmetric matrix $A$, we define $\lambda(A)$ to be the eigenvalues of $A$. We define $\lambda^{-1}(\Gamma')$ to be the space of symmetric matrices with eigenvalues in $\Gamma'$. In addition, $\lambda^{-1}(\Gamma' \cap \Gamma_+^n)$ will denote the space of positive-definite symmetric matrices with eigenvalues in $\Gamma'$. 

Lemma~\ref{lemma unif parabolic} shows that if \eqref{unif distance} holds then $\{\lambda \in \Gamma' : |\lambda| = 1\}$ is a compact subset of $\Gamma$. Since $\gamma$ is homogeneous of degree one we conclude that if \eqref{unif distance} holds then there is a constant $C = C(n,\gamma,\Gamma')$ such that
    \[C^{-1}|\xi|^2 \leq \dot\gamma^{ij}(A)\xi_i\xi_j \leq C|\xi|^2, \qquad \gamma(A)\ddot\gamma^{ij,kl}(A)S_{ij}S_{kl} \geq - C|S|^2\]
for every $A \in \lambda^{-1}(\Gamma')$, $\xi \in \mathbb{R}^n$ and $S \in \Sym(n)$. 

\begin{lemma}\label{strict iff unif}
Suppose \eqref{unif distance} holds. There is then a positive constant $\kappa = \kappa(n,\gamma,\Gamma')$ such that 
    \begin{equation}\label{unif inverse-concave}
    \gamma(A)(\ddot\gamma^{ij,kl}(A) + 2\dot\gamma^{ik}(A)A^{-1}_{jl})S_{ij}S_{kl} \geq \kappa |S|^2
    \end{equation}
for every $A \in \lambda^{-1}(\Gamma' \cap \Gamma_+^n)$ and $S \in \Sym(n)$.
\end{lemma}
\begin{proof}
Let us define
    \[\kappa := \inf_{A \in \lambda^{-1}(\Gamma'\cap\Gamma_+^n), \, S \in \operatorname{Sym}(n)} |S|^{-2}\gamma(A)(\ddot\gamma^{ij,pq}(A) + 2 \dot \gamma^{ip}(A)A^{-1}_{jq})S_{ij}S_{pq}.\]
We claim that $\kappa$ is positive.

Let $A^{(k)} \in \lambda^{-1}(\Gamma'\cap\Gamma_+^n)$ and $S^{(k)} \in \operatorname{Sym}(n)$ be sequences such that 
    \[|S^{(k)}|^{-2}\gamma(A^{(k)})(\ddot\gamma^{ij,pq}(A^{(k)}) + 2 \dot \gamma^{ip}(A^{(k)})(A^{(k)})^{-1}_{jq})S^{(k)}_{ij}S^{(k)}_{pq} \to \kappa\]
as $k \to \infty$. Since $\gamma$ is homogeneous of degree one, we may assume without loss of generality that $|A^{(k)}| = 1$ and $|S^{(k)}| = 1$. Since $\gamma$ is $O(n)$-invariant we may also assume $A^{(k)}$ is diagonal. By passing to a subsequence we can arrange that $A^{(k)}$ converges to some $\tilde A$, and that $S^{(k)}$ converges to some $\tilde S$. If $\lambda(\tilde A) \in \Gamma_+^n$ then we are done---since $\gamma$ is strictly inverse-concave, it then follows that $\kappa > 0$. Suppose instead that $\lambda(A) \in \Gamma_+^{m}$ for some $1 \leq m \leq n-1$. Since $A^{(k)}$ in $\lambda^{-1}(\Gamma')$, \eqref{unif distance} implies $m \geq m_{\IC}$. Therefore, the function $\gamma_m$ is strictly inverse-concave. 

To ease notation, let us write $A = A^{(k)}$ and $S = S^{(k)}$. Let $X$ be the $n\times n$-matrix whose entries $X_{ij} = S_{ij}$ when $i,j\leq n-m$ and vanish otherwise. Let $Z$ be the $n\times n$-matrix whose entries $Z_{ij} = S_{ij}$ when $i,j > n-m$ and vanish otherwise. We then define $Y := S - X - Z$. By Lemma~\ref{lemma unif parabolic} there is a constant $C$ depending only on $n$, $\gamma$ and $\Gamma'$ such that 
    \begin{align*}
        \ddot\gamma^{ij,kl}(A)S_{ij}S_{pq} \geq - C.
    \end{align*}
In addition, writing $\lambda = \lambda(A)$, we have
    \begin{align*} 
    2 \gamma^{ip}(A)A^{-1}_{jq} S_{ij}S_{pq} &= 2\sum_{i,\,j} \frac{\dot\gamma^i(\lambda)}{\lambda_j}|S_{ij}|^2\\
    &= 2\sum_{i, \, j \leq n - m} \frac{\dot\gamma^i(\lambda)}{\lambda_j} |X_{ij}|^2 + 2\sum_{i \leq n-m, \, j > n-m}  \frac{\dot\gamma^i(\lambda)}{\lambda_j} |Y_{ij}|^2\\
    &+2\sum_{i > n-m, \, j \leq n-m}  \frac{\dot\gamma^i(\lambda)}{\lambda_j} |Y_{ij}|^2 + 2\sum_{i, \, j > n-m} \frac{\dot\gamma^i(\lambda)}{\lambda_j} |Z_{ij}|^2.
    \end{align*}
Since $\lambda_i \to 0$ for every $1 \leq i \leq n-m$, and $\kappa < \infty$, we conclude that $|X| \to 0$ and $|Y| \to 0$ as $k \to \infty$. In particular, $\tilde S_{ij} = 0$ unless $i,j > n-m$.

To finish we combine the inequalities 
    \[\ddot\gamma^{ij,pq}(A)S_{ij}S_{pq} \geq \ddot\gamma^{ij,pq}(A) Z_{ij}Z_{pq} -C(|X|^2 + |Y|^2 + |X||Y| + |X||Z| + |Y||Z|)\]
and 
    \[2 \gamma^{ip}(A)A^{-1}_{jq} S_{ij}S_{pq}  \geq 2\sum_{i, \, j >n-m} \frac{\dot\gamma^i(\lambda)}{\lambda_j} |Z_{ij}|^2 = 2\dot\gamma^{ip}(A)A^{-1}_{jq}Z_{ij}Z_{pq}\]
in order to obtain 
    \begin{align*}
        (\ddot\gamma^{ij,pq}(A) + 2 \gamma^{ip}(A)A^{-1}_{jq})& S_{ij}S_{pq}\\
        & \geq (\ddot\gamma^{ij,pq}(A)+2 \gamma^{ip}(A)A^{-1}_{jq})Z_{ij}Z_{pq}\\
        &-C(|X|^2 + |Y|^2 + |X||Y| + |X||Z| + C|Y||Z|).
    \end{align*}
Let us write $\hat A$ and $\hat S$ for the $m\times m$-matrices which coincide with the lower-right $m \times m$-blocks of $\tilde A$ and $\tilde S$, respectively. Sending $k \to \infty$ in the last inequality then yields
    \[\kappa = (\ddot\gamma^{ij,pq}_m(\hat A)+2 \gamma^{ip}_m(\hat A)\hat A^{-1}_{jq})\hat S_{ij}\hat S_{pq}.\]
Since $\gamma_m$ is strictly inverse-concave and $|\hat S| = 1$ we conclude that $\kappa > 0$. 
\end{proof}

Next we use Lemma~\ref{strict iff unif} to show that \eqref{unif distance} implies \eqref{pert inverse-concave}.

\begin{lemma}\label{lemma pert}
Suppose \eqref{unif distance} holds. There is then a positive constant $\varepsilon = \varepsilon(n,\gamma,\Gamma')$ such that
        \begin{equation}
        \gamma(A)(\ddot\gamma^{ij,kl}(A) + 2\dot\gamma^{ik}(A)(A+\varepsilon\gamma(A)I)^{-1}_{jl})S_{ij}S_{kl} \geq 0
        \end{equation}
for every $A \in \lambda^{-1}(\Gamma')$ and symmetric matrix $S$.
\end{lemma}
\begin{proof}
Consider an arbitrary $A \in \lambda^{-1}(\Gamma')$. Set $\lambda = \lambda(A)$. Let $\theta$ be a small positive constant whose value will be fixed later, and denote by $\ell$ the integer such that $\lambda_i \leq \theta \gamma(\lambda)$ for $1 \leq i \leq \ell$ and $\lambda_i > \theta \gamma(\lambda)$ for $\ell+1 \leq i \leq n$. 

Fix an arbitrary $S \in \Sym(n)$ and write $S = X + Y +Z$, where $X_{ij} = S_{ij}$ for $i,j\leq \ell$ and $X_{ij} = 0$ otherwise, and $Z_{ij} = S_{ij}$ for $i,j > \ell$ and $Z_{ij} = 0$ otherwise.

We may assume without loss of generality that $A$ is diagonal and $\gamma(A) = 1$. We then have 
    \begin{align*}
        2\dot \gamma^{ip}(A)(A+\varepsilon&I)^{-1}_{jq}S_{ij}S_{pq}\\
        &= 2\sum_{i,j}\frac{\dot\gamma^i(\lambda)}{\lambda_j + \varepsilon } |S_{ij}|^2\\
        & = \sum_{i,j\leq \ell} \frac{\dot\gamma^i(\lambda)}{\lambda_j + \varepsilon } |X_{ij}|^2 + 2 \sum_{i \leq \ell, \, j > \ell}\frac{\dot\gamma^i(\lambda)}{\lambda_j + \varepsilon }|Y_{ij}|^2\\
        &+2\sum_{i>\ell,\,j\leq \ell}\frac{\dot\gamma^i(\lambda)}{\lambda_j + \varepsilon }|Y_{ij}|^2 + 2 \sum_{i,j > \ell}\frac{\dot\gamma^i(\lambda)}{\lambda_j + \varepsilon }|Z_{ij}|^2.
        \end{align*}
By Lemma~\ref{lemma unif parabolic} we have $\dot\gamma^i(\lambda) \geq C^{-1}$ for some positive $C = C(n,\gamma, \Gamma')$. Therefore,
    \begin{align*}
    2 \sum_{i,j > \ell}\frac{\dot\gamma^i(\lambda)}{\lambda_j + \varepsilon }|Z_{ij}|^2 &= 2 \sum_{i,j > \ell}\frac{\dot\gamma^i(\lambda)}{\lambda_j}|Z_{ij}|^2 - 2 \varepsilon \sum_{i,j > \ell}\frac{\dot\gamma^i(\lambda)}{\lambda_j(\lambda_j + \varepsilon) }|Z_{ij}|^2\\
    &\geq 2\dot \gamma^{ip}(A)(A)^{-1}_{jq}Z_{ij}Z_{pq} -C\varepsilon\theta^{-1}(\theta+\varepsilon)^{-1}|Z|^2,
    \end{align*}
and hence
    \begin{align*}
        2\dot \gamma^{ip}(A)&(A+\varepsilon I)^{-1}_{jq}S_{ij}S_{pq}\\
        &\geq  2\dot \gamma^{ip}(A)(A)^{-1}_{jq}Z_{ij}Z_{pq} + C^{-1}(\theta + \varepsilon)^{-1} (|X|^2 + |Y|^2)\\
        &-C\varepsilon\theta^{-1}(\theta+\varepsilon)^{-1}|Z|^2.
    \end{align*}
Combining this inequality with 
    \[\ddot\gamma^{ij,pq}(A)S_{ij}S_{pq} \geq \ddot\gamma^{ij,pq}(A)Z_{ij}Z_{pq} - C(|X|^2 + |Y|^2),\]
we obtain 
    \begin{align*}
        (\ddot\gamma^{ij,pq}(A) + 2&\dot \gamma^{ip}(A)(A+\varepsilon I)^{-1}_{jq})S_{ij}S_{pq}\\
        &\geq (\ddot\gamma^{ij,pq}(A) + 2\dot \gamma^{ip}(A)(A)^{-1}_{jq})Z_{ij}Z_{pq} -C\varepsilon\theta^{-1}(\theta+\varepsilon)^{-1}|Z|^2\\
        &+ (C^{-1}(\theta + \varepsilon)^{-1} - C) (|X|^2 + |Y|^2).
    \end{align*}
Lemma~\ref{strict iff unif} now gives 
    \begin{align*}
        (\ddot\gamma^{ij,pq}&(A) + 2\dot \gamma^{ip}(A)(A+\varepsilon I)^{-1}_{jq})S_{ij}S_{pq}\\
        &\geq (\kappa - C\varepsilon\theta^{-1}(\theta+\varepsilon)^{-1})|Z|^2 + (C^{-1}(\theta + \varepsilon)^{-1} - C) (|X|^2 + |Y|^2).
    \end{align*}
The right-hand side is nonnegative for $\theta < C^{-2}/2$ and $\varepsilon < \min\{ C^{-1}\kappa \theta^2, \theta\}$. 
\end{proof}

To conclude this section we show that for $k = n - m_{\IC} + 1$, uniform $k$-positivity implies \eqref{unif distance}. The converse is also true. We write 
    \[\trace(\lambda) = \lambda_1 + \dots + \lambda_n.\]

\begin{lemma}\label{lemma unif k-convex}
Let $k = n - m_{\IC} + 1$. The condition \eqref{unif distance} holds if and only if
    \begin{equation}\label{unif k-convex}
    \inf_{\lambda \in \Gamma'} \bigg(\min_{i_1 < \dots < i_k} \frac{\lambda_{i_1} + \dots + \lambda_{i_k}}{\trace(\lambda)}\bigg) > 0.
    \end{equation}
\end{lemma}
\begin{proof}
Suppose first that \eqref{unif k-convex} holds. Let us define 
    \[\delta = \inf_{0 \leq m < m_{\IC}} \bigg(\inf_{\lambda \in \Gamma'} \dist(\tfrac{\lambda}{|\lambda|}, \Gamma_+^m)\bigg).\]
The claim is that $\delta > 0$. If, to the contrary, $\delta = 0$, then there is a point $\lambda \in \Gamma'$ such that $|\lambda| = 1$ and $\lambda \in \Gamma_+^m$ for some $1 \leq m < m_{\IC}$. We may assume $\lambda_1 \leq \dots \leq \lambda_n$, in which case $\lambda \in \Gamma_+^m$ implies 
    \[0 = \lambda_1 = \dots = \lambda_{n-m}, \qquad 0 < \lambda_{n-m+1} \leq \dots \leq \lambda_n.\]
On the other hand, by \eqref{unif k-convex}, we have
    \[\lambda_1 + \dots + \lambda_k > 0.\]
Since $k = n - m_{\IC} + 1 < n - m + 1$, this is a contradiction. Therefore, we must have $\delta > 0$.

Now suppose \eqref{unif distance} holds. Let us define 
    \[\rho = \inf_{\lambda \in \Gamma'} \bigg(\min_{i_1 < \dots < i_k} \frac{\lambda_{i_1} + \dots + \lambda_{i_k}}{\trace(\lambda)}\bigg).\]
The claim is that $\rho > 0$. Assume to the contrary $\rho = 0$. There is then a point $\lambda \in \Gamma'$ such that $\trace(\lambda) = 1$, $\lambda_1 \leq \dots \leq \lambda_n$, and 
    \[0 = \lambda_{1} = \dots = \lambda_k.\]
Using $k = n-m_{\IC}+1$ we conclude that $\lambda \in \Gamma_+^m$ for some $m < m_{\IC}$. This contradicts \eqref{unif distance}, so we must have $\rho > 0$. 
\end{proof}

\section{Evolution of the Harnack quantity}\label{sec evolution}

Let $F : M\times(0,T] \to \mathbb{R}^{n+1}$ a solution to \eqref{flow}. Let $g = g(t)$ denote the metric on $M$ induced by $F(\cdot,t)$. We define a time derivative acting on vector fields by
    \[\nabla_t X^i := \partial_t X^i - GA^i_j X^j,\]
and extend to tensors via the usual Leibniz rule. This time derivative has the property that $\nabla_t g = 0$. 

Given an orthonormal frame $\{e_i\}$ of tangent vectors to $M$, we write 
    \[\Delta_\gamma  = \dot\gamma^{ij}\nabla_i \nabla_j, \qquad |A|^2_\gamma = \dot\gamma^{ij}A^2_{ij} = \dot\gamma^{ij}A_{ik}A_{kj}.\]
The first variation formula for the second fundamental form asserts that
    \[\nabla_t A_{kl} = \nabla_k \nabla_l G + G A^2_{kl}.\]
Simons' identity then yields the parabolic equation
    \[(\nabla_t - \Delta_\gamma)A_{kl} = |A|^2_\gamma A_{kl} + \ddot\gamma^{ij,pq}\nabla_k A_{ij}\nabla_l A_{pq}.\]
Since $\gamma$ is homogeneous of degree one, tracing this formula with respect to $\dot\gamma^{kl}$ gives
    \[(\partial_t - \Delta_\gamma) G = |A|^2_\gamma G.\]

The above evolution equations, together with the Gauss and Codazzi relations, imply the following commutation identities
    \begin{align*}
        \Delta_\gamma \nabla_k f &= \nabla_k \Delta_\gamma f + GA_{kl} \nabla^l f - \dot\gamma^{ij}A_{ik}A_{jl}\nabla_l f - \ddot\gamma^{ij,pq}\nabla_i\nabla_j f\nabla_k A_{pq}
    \end{align*}
and
    \begin{align*}
        \partial_t \Delta_\gamma f &= \Delta_\gamma \partial_t f + 2G\dot\gamma^{ij}A_{ik}\nabla_k\nabla_j f + 2\dot\gamma^{ij}A_{jk}\nabla_i G\nabla_k f\\
        &+\ddot\gamma^{ij,kl}\nabla_i\nabla_j f \nabla_t A_{kl}.
    \end{align*}
Using these, straightforward computations yield
    \begin{align}\label{time G}
        (\partial_t - \Delta_\gamma)\partial_t G &= |A|^2_\gamma \partial_t G + 2G \dot\gamma^{ij}A_{ik}(2\nabla_t A_{kj} - GA^2_{jk})\notag\\
        &+ 2\dot\gamma^{ij}A_{jk}\nabla_i G\nabla_k G+\ddot\gamma^{ij,pq}\nabla_t A_{ij}\nabla_t A_{pq}
    \end{align}
and
    \begin{align}\label{space G}
        (\nabla_t - \Delta_\gamma)\nabla_kG 
        &=|A|^2_\gamma \nabla_k G + \dot\gamma^{ij}A_{ik}A_{jl}\nabla_l G + 2\dot\gamma^{ij}A_{il}\nabla_k A_{lj} G \notag\\
        &+\ddot\gamma^{ij,pq}\nabla_t A_{ij}\nabla_k A_{pq}.
    \end{align}

In case $A > 0$, one finds that
    \begin{align*}
        (\nabla_t - \Delta_\gamma)A^{-1}_{kl}&= -|A|^2_\gamma A^{-1}_{kl}\\
        &-(\ddot\gamma^{ij,pq} + 2\dot\gamma^{ip}A^{-1}_{jq}) A^{-1}(e_k, \nabla A_{ij})A^{-1}(e_l, \nabla A_{pq}).
    \end{align*}
We combine this formula with \eqref{space G}
and the identity
    \begin{align*}
        - 4\dot\gamma^{ij}\nabla_i A^{-1}(\nabla G, \nabla_j \nabla G) 
        &=4\dot\gamma^{ip}A^{-1}_{jq} \nabla_t A_{ij} A^{-1}(\nabla G, \nabla A_{pq}) \\
        &- 4G\dot\gamma^{ij}A_{il} A^{-1}(\nabla G, \nabla A_{lj}),
    \end{align*}
in order to derive
    \begin{align*}
        (\partial_t &- \Delta_\gamma)(A^{-1}(\nabla G, \nabla G))\\
        &= |A|^2_\gamma A^{-1}(\nabla G, \nabla G)  - 2\dot\gamma^{ij}A^{-1}(\nabla_i \nabla G, \nabla_j \nabla G)+ 2\dot\gamma^{ij}A_{jk}\nabla_i G\nabla_kG\\
        &+ 2(\ddot\gamma^{ij,pq} + 2\dot\gamma^{ip}A^{-1}_{jq}) \nabla_t A_{ij}A^{-1}(\nabla G, \nabla A_{pq})\\
        & -(\ddot\gamma^{ij,pq} + 2\dot\gamma^{ip}A^{-1}_{jq}) A^{-1}(\nabla G, \nabla A_{ij})A^{-1}(\nabla G, \nabla A_{pq}).
    \end{align*}
Combining this with \eqref{time G} and simplifying, we obtain
    \begin{align*}
        (\partial_t - \Delta_\gamma)&(\partial_t G - A^{-1}(\nabla G, \nabla G))\\
        &=|A|^2_\gamma(\partial_t G - A^{-1}(\nabla G, \nabla G))\\
        &+(\ddot\gamma^{ij,pq} + 2\dot\gamma^{ip}A^{-1}_{jq})(\nabla_t A_{ij} - A^{-1}(\nabla G, \nabla A_{ij}))(\nabla_t A_{pq} - A^{-1}(\nabla G, \nabla A_{pq})).
    \end{align*}
Consequently,
    \begin{align*}
        (\partial_t - \Delta_\gamma)&\bigg(\partial_t G - A^{-1}(\nabla G, \nabla G) + \frac{1}{2t}G\bigg)\\
        &=\bigg(|A|^2_\gamma - \frac{2}{t}\bigg)\bigg(\partial_t G - A^{-1}(\nabla G, \nabla G) + \frac{1}{2t}G\bigg)\\
        &+ \frac{2}{t}(\partial_t G - A^{-1}(\nabla G, \nabla G)) + \frac{1}{2t^2}G\\
        &+(\ddot\gamma^{ij,pq} + 2\dot\gamma^{ip}A^{-1}_{jq})(\nabla_t A_{ij} - A^{-1}(\nabla G, \nabla A_{ij}))(\nabla_t A_{pq} - A^{-1}(\nabla G, \nabla A_{pq})).
    \end{align*}

In case $A > 0$ and $\gamma$ is inverse-concave, we have
    \begin{align*}
        (\ddot\gamma^{ij,pq} + 2&\dot\gamma^{ip}A^{-1}_{jq})(\nabla_t A_{ij} - A^{-1}(\nabla G, \nabla A_{ij}))(\nabla_t A_{pq} - A^{-1}(\nabla G, \nabla A_{pq})) \\
        &\geq \frac{2}{G}(\partial_t G - A^{-1}(\nabla G, \nabla G))^2.
    \end{align*}
Inserting this above yields 
    \begin{align}\label{Harnack scalar}
        (\partial_t - \Delta_\gamma)&\bigg(\partial_t G - A^{-1}(\nabla G, \nabla G) + \frac{1}{2t}G\bigg)\notag\\
        &\geq\bigg(|A|^2_\gamma - \frac{2}{t}\bigg)\bigg(\partial_t G - A^{-1}(\nabla G, \nabla G) + \frac{1}{2t}G\bigg)\notag\\
        &+\frac{2}{G}\bigg(\partial_t G - A^{-1}(\nabla G, \nabla G) + \frac{1}{2t}G\bigg)^2.
    \end{align}
Andrews found this inequality in \cite{Andrews_Harnack} by working in the Gauss map parameterization (this simplifies the computation considerably). By applying the parabolic maximum to \eqref{Harnack scalar}, Andrews could prove the following Harnack estimate for compact, strictly convex solutions of \eqref{flow}: at each point in spacetime one has
    \[\partial_t G - A^{-1}(\nabla G, \nabla G) + \frac{1}{2t}G \geq 0.\]
This implies \eqref{Harnack V}. 

The inequality \eqref{Harnack scalar} does not seem to be suitable for proving a Harnack inequality for noncompact solutions---various difficulties arise from the fact that $A^{-1}(\nabla G, \nabla G)$ might be unbounded. To prove Theorem~\ref{Harnack}, we modify Hamilton's approach to the Harnack estimate for mean curvature flow \cite{Hamilton_Harnack_MCF}. That is, we define a form $Q$ acting on tangent vectors by
    \[Q(V) := \partial_t G + 2\langle \nabla G, V \rangle + A(V,V) + \frac{1}{2t}G,\]
and use the maximum principle to conclude that $Q$ is nonnegative. The key advantage of this is that at a spacetime minimum of $Q$, given any minimizing vector $V$, there is some freedom in how we extend $V$ to a neighbourhood. This freedom can be exploited to introduce terms which are favourable, but suitably controlled---see Remark~\ref{need for uniformity}. 

The rest of the computations in this section apply to a general solution of \eqref{flow}.

Let $V$ be a time-dependent field of tangent vectors on $M$. Following \cite{Hamilton_Harnack_MCF}, we introduce tensors
    \begin{align*}
        X_i &= \nabla_i G + A_{ij} V^j\\
        Y_{ij} &= \nabla_i V_j - GA_{ij} - \frac{1}{2t}g_{ij}\\
        W_{ij} &= \nabla_t A_{ij} + V^k \nabla_k A_{ij} + \frac{1}{2t}A_{ij}\\
        U^i &= (\nabla_t - \Delta_\gamma)V^i + \dot \gamma^{ij} A_{jk}\nabla_k G + \frac{1}{t}V^i.
    \end{align*}
We compute at a point $(x_0, t_0)$ in spacetime with respect to a local orthonormal frame for which $\nabla_i e_j = 0$ and $\nabla_t e_i = 0$ at $(x_0, t_0)$. 

\begin{proposition}\label{Harnack evolution}
At the point $(x_0,t_0)$ we have 
    \begin{align}\label{Harnack tensor}
        (\partial_t &- \Delta_\gamma)(Q(V))\notag\\
        &= \bigg(|A|^2_\gamma - \frac{2}{t}\bigg) Q(V) + 2X_iU_i - 4\dot\gamma^{ij}Y_{ik}W_{jk}-2\dot\gamma^{ij}Y_{ik}A_{kl}Y_{jl}\notag\\
        &+ \ddot\gamma^{ij,pq}(\nabla_t A_{ij} + V_k\nabla_kA_{ij})(\nabla_tA_{pq} + V_l\nabla_lA_{pq}).
    \end{align}
\end{proposition}
\begin{proof}
A nice account of Hamilton's computation for the mean curvature flow is given in \cite[Proposition~10.7]{Andrews_etc}. We follow the exposition there. 

Note that 
    \[Q(V) = \dot\gamma^{ij}W_{ij} + X_iV_i.\]
Using \eqref{space G} and the first variation formula 
    \[\nabla_t A_{ij} = \nabla_i \nabla_j G + GA^2_{ij},\]
we compute
    \begin{align*}
        (\partial_t - \Delta_\gamma)(X_k V_k)&=|A|^2_\gamma X_kV_k + (X_k + A_{ik}V_i)(\nabla_t - \Delta_\gamma)V_k+2G\dot\gamma^{ij}A_{il}\nabla_kA_{lj}V_k\\
        &+ \dot\gamma^{ij}A_{il}A_{jk}V_k \nabla_l G - 2\dot\gamma^{ij}(\nabla_i X_k + V_l\nabla_iA_{kl})\nabla_j V_k\\
        &+\ddot\gamma^{ij,pq}V_k\nabla_k A_{ij}V_l\nabla_l A_{pq} + \ddot\gamma^{ij,pq}\nabla_t A_{ij} V_k\nabla_k A_{pq}.
    \end{align*}
Next we use \eqref{time G} to derive
    \begin{align*}
        (\partial_t - \Delta_\gamma)(\dot\gamma^{ij}W_{ij}) &= |A|^2_\gamma \dot\gamma^{ij}W_{ij} + \nabla_i G(\nabla_t - \Delta_\gamma)V_i + \dot\gamma^{ij}A_{il}A_{jk}\nabla_l G V_k\\
        &+2G\dot\gamma^{ij}A_{ik}(2\nabla_t A_{jk} - GA^2_{jk} + V_l\nabla_lA_{jk})\\
        &- 2\dot\gamma^{ij}\nabla_iV_k\nabla_j\nabla_k G +2\dot\gamma^{ij}A_{jk}\nabla_iG\nabla_kG - \frac{1}{2t^2}G\\
        &+\ddot\gamma^{ij,pq}\nabla_t A_{ij}\nabla_t A_{pq} + \ddot\gamma^{ij,pq}V_k\nabla_kA_{ij}\nabla_tA_{pq}.
    \end{align*}
Summing these two formulae yields
    \begin{align*}
        (\partial_t - \Delta_\gamma)(Q(V)) &= |A|^2_\gamma Q(V) + 2X_i(\nabla_t - \Delta_\gamma)V_i\\
        &+2G\dot\gamma^{ij}A_{ik}(2\nabla_t A_{jk} - GA^2_{jk} + 2V_l\nabla_lA_{jk})\\
        &-2\dot\gamma^{ij}\nabla_i V_k(2\nabla_t A_{jk} - 2GA^2_{jk} + 2\nabla_kA_{jl}V_l + A_{kl}\nabla_jV_l)\\
        &+2\dot\gamma^{ij}A_{il}A_{jk}V_k\nabla_lG + 2\dot\gamma^{ij}A_{jk}\nabla_i G \nabla_kG - \frac{1}{2t^2}G\\
        &+\ddot\gamma^{ij,pq}(\nabla_t A_{ij} + V_k\nabla_kA_{ij})(\nabla_tA_{pq} + V_l\nabla_lA_{pq}).
    \end{align*}
We may rewrite this as 
    \begin{align*}
        (\partial_t &- \Delta_\gamma)(Q(V))\\
        &= \bigg(|A|^2_\gamma - \frac{2}{t}\bigg) Q(V) + 2X_iU_i - 4\dot\gamma^{ij}Y_{ik}W_{jk}-2\dot\gamma^{ij}\nabla_iV_kA_{kl}Y_{jl}\\
        &- 2G\dot\gamma^{ij}\bigg(\frac{1}{2t}A_{jk} + GA^2_{jk}\bigg)+2\dot\gamma^{ij}\nabla_iV_k\bigg(\frac{1}{2t}A_{jk} + GA^2_{jk}\bigg)\\
        &+ 2\dot\gamma^{ij}A_{ik}\nabla_kG(A_{jl}V_l + \nabla_j G - X_j) - \frac{1}{t} \dot\gamma^{ij}\bigg(\frac{1}{2t}A_{ij} + GA^2_{ij}\bigg)\\
        &+ \ddot\gamma^{ij,pq}(\nabla_t A_{ij} + V_k\nabla_kA_{ij})(\nabla_tA_{pq} + V_l\nabla_lA_{pq}).
    \end{align*}
The first term on the penultimate line is zero. Collecting the remaining terms, the claim follows.
\end{proof}

Notice that if $A > 0$ then \eqref{Harnack scalar} is recovered when we set $V = - A^{-1}(\nabla G)$ in \eqref{Harnack tensor}.

\section{Application of the maximum principle}

We now prove Theorem~\ref{Harnack}. The function $\gamma$ is assumed to be inverse-concave (when $\gamma$ is convex the proof is analogous but much simpler). Let $F : M\times[0,T] \to \mathbb{R}^{n+1}$ be a complete solution to \eqref{flow} such that $A \geq 0$. We assume that \eqref{unif ellipticity} and \eqref{pert inverse-concave} hold on $M\times[0,T]$ for some positive constants $C$ and $\varepsilon$. That is, with respect to any orthonormal frame,
    \[C^{-1}g_{ij} \leq \dot\gamma^{ij} \leq Cg_{ij},\]
and 
    \begin{equation*}\label{pert inverse-concave local}
    (\ddot\gamma^{ij,kl} + 2\dot\gamma^{ik}(A+ \varepsilon G g)^{-1}_{jl})S_{ij}S_{kl}\geq 0
    \end{equation*}
for every symmetric $S$. We also assume bounded curvature,
    \[\Xi := \sup_{M \times [0, T]} G < \infty,\]
and pointwise gradient estimates of the form \eqref{gradient estimate},
    \[\Lambda := \sup_{M\times[0,T]} G^{-2}|\nabla A| + G^{-3}|\nabla^2 A| < \infty.\]
    
We define
    \[\hat Q(V) = Q(V) + \varphi + \psi|V|^2\]
on $M\times(0,T]$, where $\varphi(x,t)$ and $\psi(t)$ are functions to be chosen later. For now we only assume that $\varphi$ and $\psi$ are bounded from below by positive constants, and that $\varphi$ grows at spatial infinity in the following sense: for some $p_0 \in M$,
    \begin{equation}\label{phi growth}
        \lim_{R \to \infty} \inf_{(x,t)\in M\setminus B_{g(0)}(p_0, R) \times (0,T]} \varphi(x,t) = \infty.
    \end{equation}
Since $A \geq 0$ and the remaining terms in $Q$ are bounded, these assumptions ensure that $\hat Q > 0$ on $M$ at times close to zero.

Suppose, with the aim of deriving a contradiction, that $\hat Q$ fails to be positive on $M\times[0,T]$. Then there exists a spacetime point $(x_0, t_0)$ such that $\hat Q > 0$ for $0 < t < t_0$ and $\hat Q(V) = 0$ at $(x_0, t_0)$ for some $V \in T_{x_0} M$. For any extension of $V$, Proposition~\ref{Harnack evolution} shows that at $(x_0, t_0)$ we have 
    \begin{align*}
        0 &\geq - 4\dot\gamma^{ij}Y_{ik}W_{jk}-2\dot\gamma^{ij}Y_{ik}A_{kl}Y_{jl}\notag\\
        &+ \ddot\gamma^{ij,pq}(\nabla_t A_{ij} + V_k\nabla_kA_{ij})(\nabla_tA_{pq} + V_l\nabla_lA_{pq})\\
        &+(\partial_t - \Delta_\gamma)\varphi + \partial_t \psi |V|^2 + 2\psi V_i(\nabla_t - \Delta_\gamma)V_i\\
        &-2\psi\dot\gamma^{ij}\nabla_iV_k\nabla_jV_k + \bigg(\frac{2}{t}-|A|^2_\gamma\bigg)(\varphi + \psi |V|^2).
    \end{align*}    

\begin{remark}\label{need for uniformity} To derive a contradiction we would like to choose $\varphi$ and $\psi$, and an extension of $V$, so that the right-hand side of the last inequality is positive. Given any matrix $B$, we are free to extend $V$ so that, at the point $(x_0,t_0)$,
    \[Y_{ij} = B_{ij}, \qquad U = 0.\]
In case $\gamma$ is convex it is sufficient to choose $B = 0$. If $A > 0$ and $\gamma$ is inverse-concave, choosing $B_{ij} = -W_{ik}A^{-1}_{kj}$ results in a positive term which overcomes the term involving second derivatives of $\gamma$. But when the solution is noncompact, $A^{-1}$ is unbounded, and this introduces errors which cannot be absorbed using $\varphi$ and $\psi$. This is why we assume \eqref{unif ellipticity}, \eqref{pert inverse-concave} and \eqref{gradient estimate}. We extend $V$ so that 
    \begin{equation}\label{choice Y}
        Y_{ij} = -W_{ik}(A+\varepsilon G g)^{-1}_{kj}
    \end{equation}
at $(x_0,t_0)$, where $\varepsilon$ is the constant in \eqref{pert inverse-concave}. This is sufficient to overcome the term involving second derivatives of $\gamma$, and only introduces errors which are bounded in terms of $C$, $\Xi$, $\Lambda$ and $\varepsilon$.
\end{remark}

In what follows $\Theta$ is a large constant which may depend on $n$, $\gamma$, $C$, $\Xi$ and $\Lambda$. The value of $\Theta$ may change from line to line. 

Inserting \eqref{choice Y}, we obtain
    \begin{align*}- 4\dot\gamma^{ij}Y_{ik}W_{jk} - 2\dot\gamma^{ij}Y_{ik}A_{kl}Y_{jl}=2\dot\gamma^{ik}(A+\varepsilon G g)^{-1}_{jl}W_{ij}W_{kl} + 2\varepsilon G \dot\gamma^{ij}Y_{ik}Y_{jk}.
    \end{align*}
Since $\gamma$ is inverse-concave and homogeneous of degree one, if $A$ is positive-definite at $(x_0,t_0)$ then we have
    \[\ddot\gamma^{ij,pq}A_{ij}S_{pq} =-(2\dot\gamma^{ip}A^{-1}_{jq} - 2G^{-1}\dot\gamma^{ij}\dot\gamma^{pq})A_{ij}S_{pq} = 0\]
for every symmetric matrix $S$. Actually the same holds even if $A$ is only nonnegative---to see this, approximate $A_{ij}$ by a sequence of positive-definite symmetric matrices and pass to the limit. It follows that 
    \begin{align*}
        \ddot\gamma^{ij,pq}&(\nabla_t A_{ij} + V_k\nabla_kA_{ij})(\nabla_t A_{pq} + V_l\nabla_l A_{pq})=\ddot\gamma^{ij,pq}W_{ij}W_{pq},
    \end{align*}
and hence 
    \begin{align*}
        \ddot\gamma^{ij,pq}&(\nabla_t A_{ij} + V_k\nabla_kA_{ij})(\nabla_t A_{pq} + V_l\nabla_l A_{pq})- 4\dot\gamma^{ij}Y_{ik}W_{jk} - 2\dot\gamma^{ij}A_{jl}Y_{ki}Y_{kl}\\
        &= (\ddot\gamma^{ij,pq} + 2\dot\gamma^{ip}(A+\varepsilon G g)^{-1}_{jq})W_{ij}W_{pq} + 2\varepsilon G \dot\gamma^{ij}Y_{ik}Y_{jk}.
    \end{align*}
The right-hand side is nonnegative by hypothesis. Moreover, if we extend $V$ so that $U = 0$ at $(x_0, t_0)$, then 
    \[2\psi V_i(\nabla_t - \Delta_\gamma)V_i \geq - \Theta\psi|V| - \frac{2}{t}\psi|V|^2.\]
Combining these facts we find that
    \begin{align*}
        0 &\geq (\partial_t - \Delta_\gamma)\varphi + (\partial_t \psi - \Theta\psi) |V|^2 -2\psi\dot\gamma^{ij}\nabla_iV_k\nabla_jV_k+ \bigg(\frac{2}{t}-|A|^2_\gamma\bigg)\varphi - \Theta\psi
    \end{align*}
at $(x_0,t_0)$. 

We may assume without loss of generality $\varepsilon < 1$. Since 
    \[G|(A+\varepsilon G g)^{-1}| \leq \Theta \varepsilon^{-1},\] 
we have 
    \begin{align*}
        \dot\gamma^{ij}\nabla_iV_k &\nabla_jV_k \\
        &\leq \Theta(|Y|^2 + G^2|A|^2 + t^{-2})\\
        &\leq \Theta\varepsilon^{-2}G^{-2}|W|^2 + \Theta(t^{-2} + 1)\\
        &\leq \Theta \varepsilon^{-2}(G^{-2}|\nabla^2 A|^2 + G^{-2}|\nabla A|^2|V|^2 + G^{-2}|A|^6 + t^{-2}G^{-2}|A|^2)\\
        &+\Theta(t^{-2} + 1)\\
        &\leq \Theta \varepsilon^{-2}(G^{-6}|\nabla^2 A|^2 + G^{-4}|\nabla A|^2|V|^2 + t^{-2} + 1).
    \end{align*}
As a consequence of the pointwise gradient estimates, 
    \[-2\psi\dot\gamma^{ij}\nabla_iV_k\nabla_jV_k \geq - \Theta\varepsilon^{-2} \psi |V|^2 - \Theta\varepsilon^{-2}\bigg(1+\frac{1}{t^2}\bigg)\psi.\]
Therefore, at $(x_0, t_0)$ we have 
    \begin{align*}
        0 &\geq (\partial_t - \Delta_\gamma)\varphi + (\partial_t \psi - \Theta\varepsilon^{-2}\psi)|V|^2 + \bigg(\frac{2}{t}-\Theta\bigg)\varphi -\Theta\varepsilon^{-2}\bigg(1+\frac{1}{t^2}\bigg)\psi.
    \end{align*}

Let $b$, $K$ and $L$ be positive constants. For $t \in (0,T]$ set
    \[\psi(t) := b e^{Kt}, \qquad \varphi(x,t):=b t^{-3/2}e^{Lt} f(x,t),\]
where $f(x,t) \geq 1$ is such that \eqref{phi growth} holds and we have
    \[(\partial_t - \Delta_\gamma)f \geq 2Lf.\]
The construction of a function $f$ with these properties is standard---see eg. \cite[Lemma~12.7]{Chow_etc}. Inserting these definitions gives 
    \begin{align*}
        0 &\geq \bigg(\frac{1}{2t} + 2L - \Theta\bigg)\frac{be^{Lt} }{t^{3/2}} f -\Theta \varepsilon^{-2}\bigg(1+\frac{1}{t^2}\bigg) be^{Kt} + (K - \Theta \varepsilon^{-2})\psi|V|^2
    \end{align*}
at $(x_0, t_0)$. For $K > \Theta \varepsilon^{-2}$ and $L > \max\{\Theta/2, K\}$ we conclude that 
        \begin{align*}
        0 &\geq \bigg(\frac{1}{2t} + 2L - \Theta\bigg)\frac{1}{t^{3/2}} -\Theta \varepsilon^{-2}\bigg(1+\frac{1}{t^2}\bigg).
    \end{align*}
at $t = t_0$. But if $L$ is sufficiently large the quantity on the right is positive for every $t \in (0, T]$. This is a contradiction. 

We thus conclude that $\hat Q(V) > 0$ for every point $(x,t) \in M \times (0, T]$ and tangent vector $V \in T_x M$. Sending $b \to 0$, we obtain $Q(V) \geq 0$. With this the proof of Theorem~\ref{Harnack} is complete.

Theorem~\ref{Harnack k-convex} is proven as follows. Here we are assuming $\gamma$ is convex, or concave and strictly inverse-concave. Suppose $F$ is a convex ancient solution satisfying the hypotheses of the theorem. Fix a $T < \infty$. Since $F$ is assumed to be uniformly $k$-convex on $[-T,0]$, with $k \leq n - m_{\IC} + 1$, Lemma~\ref{lemma unif parabolic} implies there is a constant $C$ such that \eqref{unif ellipticity} holds for $t \in [-T, 0]$. Moreover, Lemma~\ref{lemma pert} implies there is a constant $\varepsilon$ such that \eqref{pert inverse-concave} holds for $t \in [-T, 0]$. Since we are assuming convexity and noncollapsing, Corollary~5.2 in \cite{Brendle_Huisken} (that result is for a specific flow, but the proof generalises---see \cite[Theorem~4.14]{Lynch_Thesis} for the general case) yields the pointwise gradient estimates
    \[G^{-2}|\nabla A| + G^{-3}|\nabla^2 A| \leq \Lambda\]
for some $\Lambda < \infty$. This $\Lambda$ depends only on $n$, $\gamma$, the ellipticity constant $C$ and the noncollapsing constant.  Therefore, $F$ satisfies all of the hypotheses of Theorem~\ref{Harnack} on $[-T, 0]$. Consequently, we have
    \[\partial_t G + 2\langle \nabla G, V\rangle + A(V,V) + \frac{G}{2(t+T)} \geq 0\]
at every point $(x,t) \in M\times(-T,0]$ and for every $V \in T_x M$. Since this is true for every $T < \infty$, we may send $T \to \infty$ to obtain
    \[\partial_t G + 2\langle \nabla G, V\rangle + A(V,V) \geq 0.\]
This completes the proof of Theorem~\ref{Harnack k-convex}.

\section{The equality case}
In this section we prove Theorem~\ref{translator}. Suppose $\gamma$ is strictly inverse-concave. Let $F : M\times(-\infty,0] \to \mathbb{R}^{n+1}$ be an ancient solution to \eqref{flow} which satisfies $A > 0$ and the Harnack inequality
    \[\partial_t G -A^{-1}(\nabla G, \nabla G) \geq 0\]
at every point in spacetime. In addition, we assume bounded curvature
    \[\sup_{M\times[-T,0]} G < \infty\]
and the uniform ellipticity condition~\eqref{unif ellipticity} in $[-T,0]$ for each $T < \infty$.

Suppose there is some $(x_0,t_0) \in M \times (-\infty,0]$ at which
    \[\partial_t G -A^{-1}(\nabla G, \nabla G) = 0.\]
The computations at the beginning of Section~\ref{sec evolution} show that 
    \begin{align*}
        (\partial_t &- \Delta_\gamma)(\partial_t G - A^{-1}(\nabla G, \nabla G))\\
        &\geq|A|^2_\gamma(\partial_t G - A^{-1}(\nabla G, \nabla G))\\
        &+(\ddot\gamma^{ij,kl} + 2\dot\gamma^{ik}A^{-1}_{jl})(\nabla_t A_{ij} - A^{-1}(\nabla G, \nabla A_{ij}))(\nabla_t A_{kl} - A^{-1}(\nabla G, \nabla A_{kl}))
    \end{align*}
on $M\times(-\infty,0]$. The right-hand side is nonnegative, so the strong maximum principle implies
    \[\partial_t G - A^{-1}(\nabla G, \nabla G) = 0\]
for $t \leq t_0$. Since $\gamma$ is assumed to be strictly inverse-concave, we conclude that 
    \[\nabla_t A_{ij} - A^{-1}(\nabla G, \nabla A_{ij}) = 0,\]
for $t \leq t_0$.

Define $\xi = -G\nu - A^{-1}(\nabla G)$. Using the Codazzi equations we compute
    \begin{align*}
        \langle D_i \xi, e_j\rangle &= A^{-1}_{jk} A^{-1}(\nabla G, \nabla A_{ik}) - GA_{ij} - A_{jk}^{-1}\nabla_i\nabla_k G \\
        &= A^{-1}_{jk} A^{-1}(\nabla G, \nabla A_{ik}) - A_{jk}^{-1}\nabla_t A_{ik}\\
        &=0
    \end{align*}
for $t \leq t_0$, where $D$ is the Euclidean connection on $\mathbb{R}^{n+1}$. It follows that $\xi$ can be extended to a constant vector field $\xi$ on $\mathbb{R}^{n+1}$ and, by definition,
    \[G = -\langle \xi, \nu \rangle.\]
That is, $F$ is a translating soliton with velocity $\xi$ for $t \leq t_0$. 

It is now straightforward to show that $F$ is a translating soliton with velocity $\xi$ for $t \leq 0$. Since $\xi$ is constant on $\mathbb{R}^{n+1}$, the quantity $u := G + \langle \xi, \nu \rangle$ satisfies
    \[\partial_t u = \dot\gamma^{ij}\nabla_i \nabla_j u + \dot \gamma^{ij}A_{ik}A_{kj}u.\]
Moreover, because we have bounded curvature and uniform ellipticity, the function $h(x,t) = e^{Kt}(|x|^2 + 1)$ satisfies
    \[\partial_t h > \dot\gamma^{ij}\nabla_i \nabla_j h + \dot \gamma^{ij}A_{ik}A_{kj} h\]
for $t \geq t_0$ if the constant $K$ is sufficiently large. By the maximum principle, $\sup_{M_t} \tfrac{|u|}{h}$ is nonincreasing for $t_0 \leq t \leq 0$. Since $u$ vanishes at $t = t_0$, we have $G = -\langle \xi, \nu\rangle$ for all $t \leq 0$.

\bibliography{references}
\bibliographystyle{alpha}

\end{document}